\documentclass{amsart}

\usepackage{amssymb,latexsym,amsmath,amsthm,enumitem,mathrsfs,geometry}
\usepackage{hyperref}
\usepackage{cleveref}
\usepackage{bbm}

\newtheorem{Theorem}{Theorem}
\newtheorem*{theorem*}{Theorem}
\newtheorem{corollary}{Corollary}
\newtheorem{lem}{Lemma}

\newcommand{\1}{\mathbbm{1}}

\begin{document}

\title[Average Goldbach Representation Formula of Fujii]{On an Average Goldbach Representation Formula of Fujii}

\author[Goldston]{D. A. Goldston}
 \address{Department of Mathematics and Statistics, San Jose State University}
 \email{daniel.goldston@sjsu.edu}

\author[Suriajaya]{Ade Irma Suriajaya$^{*}$}
\thanks{$^{*}$ The second author was supported by JSPS KAKENHI Grant Numbers 18K13400 and 22K13895, and also by MEXT Initiative for Realizing Diversity in the Research Environment.}
 \address{Faculty of Mathematics, Kyushu University}
 \email{adeirmasuriajaya@math.kyushu-u.ac.jp}


\maketitle

\begin{abstract}
Fujii obtained a formula for the average number of Goldbach representations with lower order terms expressed as a sum over the zeros of the Riemann zeta-function and a smaller error term. This assumed the Riemann Hypothesis. We obtain an unconditional version of this result, and obtain applications conditional on various conjectures on zeros of the Riemann zeta-function.

\end{abstract}

\bigskip

\noindent
{\bf Added in Proof.} Languasco, Perelli, and Zaccagnini \cite{LPZ12,LPZ16,LPZ17} have obtained many results connecting conjectures related to
pair correlation of zeros of the Riemann zeta function to conjectures on primes. It has been
brought to our attention that the main result in our follow-up paper
\cite{GS22} to this paper on the error in the prime number theorem has
already been obtained in \cite{LPZ16}. Our method is based on a generalization $F_\beta(x,T)$ of $F(x,T)$ from 
\eqref{Fdef} where $w(u)$ is replaced with $w_\beta(u) = \frac{4\beta^2}{4\beta^2+u^2}$. In \cite{LPZ16} they used $F_\beta(x,T)$ with a change of variable $\beta = 1/\tau$. The results we obtained are analogous to some of their results. In the paper \cite{LPZ17} it is shown how this method can be applied to generalizations of $\mathcal{J}(x,\delta)$ in \eqref{calJ}.
We direct interested readers to these papers of Languasco,
Perelli, and Zaccagnini.

\bigskip

\section{Introduction and Statement of Results}

Let
\begin{equation}\label{psi_2} \psi_2(n) = \sum_{m+m'=n} \Lambda(m)\Lambda(m'),\end{equation}
where $\Lambda$ is the von Mangoldt function, defined by $\Lambda(n)=\log p$ if $n=p^m$, $p$ a prime and $m\ge 1$, and $\Lambda(n)= 0$ otherwise. Thus $\psi_2(n)$ counts \lq \lq Goldbach" representations of $n$ as sums of both primes and prime powers, and these primes are weighted to make them have a \lq\lq density" of 1 on the integers. 
Fujii \cite{Fujii1,Fujii2,Fujii3} in 1991 proved the following theorem concerning the average number of Goldbach representations. 
\begin{theorem*}[Fujii] Assuming the Riemann Hypothesis, we have
\begin{equation} \label{FujiiThm} \sum_{n\le N} \psi_2(n) = \frac{N^2}{2} - 2\sum_{\rho} \frac{N^{\rho+1}}{\rho(\rho +1)} + O(N^{4/3}(\log N)^{4/3}), \end{equation}
where the sum is over the complex zeros $\rho=\beta +i\gamma$ of the Riemann zeta-function $\zeta(s)$, and the Riemann Hypothesis is $\beta = 1/2$. 
\end{theorem*}
Thus the average number of Goldbach representations is connected to the zeros of the Riemann zeta-function, and as we will see later, is also closely connected to the error in the prime number theorem. The sum over zeros appears in the asymptotic formula even without the assumption of the Riemann Hypothesis, but the Riemann Hypothesis is needed to estimate the error term.
With regard to the sum over zeros, it is useful to keep in mind the Riemann von Mangoldt formula
\begin{equation} \label{N(T)} N(T) := \sum_{0 < \gamma \le T} 1= \frac{T}{2\pi} \log \frac{T}{2 \pi} - \frac{T}{2\pi} + O(\log T), \end{equation}
see \cite[Theorem 25]{Ingham1932} or \cite[Theorem 9.4]{Titchmarsh}.
Thus $N(T) \sim \frac{T}{2\pi}\log T$, and we also obtain
\begin{equation} \label{N(T+1)-N(T)} N(T+1) - N(T) = \sum_{T<\gamma \le T+1} 1 \ll \log T. \end{equation}
This estimate is very useful, and it shows that the sum over zeros in \eqref{FujiiThm} is absolutely convergent.
Hence, the Riemann Hypothesis implies that
$$ \sum_{n\le N} \psi_2(n) = \frac{N^2}{2} + O(N^{3/2}). $$
This was shown by Fujii in \cite{Fujii1}.
Unconditionally, Bhowmik and Ruzsa \cite{BhowRuzsa2018} showed that the estimate
$$ \sum_{n\le N} \psi_2(n) = \frac{N^2}{2} + O(N^{2-\delta}) $$
implies that for all complex zeros $\rho=\beta+i\gamma$ of the Riemann zeta-function, we have $\beta<1-\delta/6$.
We are interested, however, in investigating the error in \eqref{FujiiThm}, that is the error estimate not including the sum over zeros.

It was conjectured by Egami and Matsumoto \cite{EgamiMatsumoto} in 2007 that the error term in \eqref{FujiiThm} above can be improved to $O(N^{1+\varepsilon})$ for any $\varepsilon>0$. That error bound was finally achieved, assuming the Riemann Hypothesis, by Bhowmik and Schlage-Puchta in \cite{BhowPuchta2010} who obtained $O(N\log^5\!N)$, and this was refined by Languasco and Zaccagnini \cite{Lang-Zac1} who obtained the following result. 
\begin{theorem*}[Languasco-Zaccagnini] Assuming the Riemann Hypothesis, we have
\begin{equation} \label{LZ} \sum_{n\le N} \psi_2(n) = \frac{ N^2}{2} -2 \sum_{\rho} \frac{N^{\rho+1} }{\rho(\rho +1)}+O( N\log^3\!N).\end{equation}
\end{theorem*}
A different proof of this theorem was given in \cite{Gold-Yang} along the same lines as \cite{BhowPuchta2010}. It was proved in \cite{BhowPuchta2010} that unconditionally
\begin{equation} \label{BPomega} \sum_{n\le N} \psi_2(n) = \frac{ N^2}{2} -2 \sum_{\rho} \frac{N^{\rho+1} }{\rho(\rho +1)}+\Omega( N\log\log N),\end{equation}
and therefore the error term in \eqref{LZ} is close to best possible. 

In this paper we combine and hopefully simplify the methods of \cite{BhowPuchta2010} and \cite{Lang-Zac1}. Our method is based on an exact form of Fujii's formula \eqref{FujiiThm} where the error term is explicitly given. We state this as Theorem \ref{thm1}.
We will relate this error term to the distributions of primes, and this can be estimated using the variance of primes in short intervals.

We follow the notation and methods of Montgomery and Vaughan \cite{MontgomeryVaughan1973} fairly closely in what follows.
Sums in the form of $\sum_\rho$ or $\sum_\gamma$ run over nontrivial zeros $\rho=\beta+i\gamma$ of the Riemann zeta function and all other sums run over the positive integers unless specified otherwise, so that $\sum_n = \sum_{n\ge 1}$ and $\sum_{n\le N} = \sum_{1\le n\le N}$. We use the power series generating function
\begin{equation} \label{Psi} \Psi(z) = \sum_{n} \Lambda(n) z^n \end{equation}
which converges for $|z|<1$, and obtain the generating function for $\psi_2(n)$ in \eqref{psi_2} directly since
\begin{equation} \Psi(z)^2 = \sum_{m,m'}\Lambda(m)\Lambda(m') z^{m+m'} = \sum_n \left(\sum_{m+m'=n}\Lambda(m)\Lambda(m')\right) z^{n}= \sum_{n} \psi_2(n) z^n. \end{equation}
Our goal is to use properties of $\Psi(z)$ to study averages of the coefficients $\psi_2(n)$ of $\Psi(z)^2$. 

Our version of Fujii's theorem is as follows. We take
\begin{equation} z=re(\alpha), \qquad e(u)= e^{2\pi i u},\end{equation}
where $0\le r < 1$, and define
\begin{equation} \label{I_N} I(r,\alpha) := \sum_n r^ne(\alpha n), \qquad I_N(r,\alpha) := \sum_{n\le N} r^ne(\alpha n). \end{equation}
We also, accordingly, write $\Psi(z) = \sum_{n} \Lambda(n) r^ne(\alpha n)$ as $\Psi(r,\alpha)$.
\begin{Theorem} \label{thm1} For $N\ge 2$, we have 
\begin{equation} \label{th1result} \sum_{n\le N} \psi_2(n) = \frac{N^2}{2} - 2\sum_\rho \frac{N^{\rho +1}}{\rho(\rho +1)}
-( 2\log 2\pi-\frac12)N + 2\frac{\zeta '}{\zeta}(-1) - \sum_k \frac{N^{1-2k}}{k(2k-1)} + E(N) , \end{equation}
where, for $0<r<1$,
\begin{equation} E(N) := \int_0^1 (\Psi(r,\alpha)-I(r,\alpha))^2 I_N(1/r,-\alpha)\, d\alpha . \end{equation} In particular, we have
\begin{equation} \label{specialcase} \sum_{n\le N} \psi_2(n) = \frac{N^2}{2} - 2\sum_\rho \frac{N^{\rho +1}}{\rho(\rho +1)}
+ E(N) +O(N). \end{equation}
\end{Theorem}

The prime number theorem is equivalent to 
\begin{equation} \label{PNT} \psi(x) := \sum_{n\le x}\Lambda (n) \sim x ,\qquad \text{as} \quad x\to \infty. \end{equation}
It was shown in \cite{BhowPuchta2010} that the error term $E(N)$ in \Cref{thm1} can be estimated using the functions\footnote[2]{ In \cite{BhowPuchta2010} and also \cite{Gold-Yang} there is a third integral error term that is also needed, but the method of \cite{Lang-Zac1} and here avoids this term.} 
\begin{equation} \label{HandJ} H(x) := \int_0^x (\psi(t) -t)^2 \, dt \qquad \text{and} \quad J(x,h) := \int_0^x (\psi(t+h) -\psi(t) -h)^2 \, dt. \end{equation}
The integral $H(x)$ was studied by Cram\'er \cite{Cramer21} in 1921; he proved assuming the Riemann Hypothesis that
\begin{equation} \label{Cramer} H(x) \ll x^2 . \end{equation}
Selberg \cite{Selberg} was the first to study variances like $J(x,h)$ and obtain results on primes from this type of variance, both unconditionally and on the Riemann Hypothesis. The estimate we need is a refinement of Selberg's result which was first obtained by Saffari and Vaughan \cite{SaffariVaughan}; they proved, assuming the Riemann Hypothesis, that for $1\le h \le x$,
\begin{equation} \label{SV} J(x,h) \ll hx\log^2(\frac{2x}{h}). \end{equation}
By a standard argument using Gallagher's lemma \cite[Lemma 1.9]{Montgomery71} we obtain the following unconditional bound for $E(N)$. 

\begin{Theorem} \label{thm2} Let $\log_2\!N$ denote the logarithm base 2 of $N$. Then for $N\ge 2$, we have
\begin{equation} \label{E(N)bound}
\begin{aligned}
|E(N)| \le \mathcal{E}(N) &:= \int_0^1 |\Psi(r,\alpha)-I(r,\alpha)|^2 |I_N(1/r,-\alpha)|\, d\alpha \\
&\ll \sum_{0\le k < \log_2\!N} \frac{N}{2^k} \mathcal{W}\left(N,\frac{N}{2^{k+2}}\right),
\end{aligned}
\end{equation}
where, for $1\le h \le N$,
\begin{equation}\label{Wbound}
\begin{aligned}
\mathcal{W}(N,h) &:= \int_0^{1/2h} \left| \sum_n ( \Lambda(n)-1) e^{-n/N} e(n\alpha) \right|^2 \, d\alpha \\
&\ll \frac1{h^2} H(h) + \frac{1}{N^2}\sum_j \frac1{2^j}H(jN) + \frac1{h^2}\sum_j \frac1{2^j}J(jN,h) + \frac{N}{h^2}.
\end{aligned}
\end{equation}
\end{Theorem}

Not only can Theorem \ref{thm2} be applied to the error term in Fujii's theorem, but it can also be used in bounding the error in the prime number theorem.

\begin{Theorem} \label{thm3} For $N\ge 2$ and $0<r<1$, we have 
\begin{equation}\label{PNTerror} \psi(N) - N = \int_0^1 (\Psi(r,\alpha)-I(r,\alpha)) I_N(1/r,-\alpha)\, d\alpha , \end{equation} 
and
\begin{equation} \label{PNTbound} \psi(N) - N \ll \sqrt{\mathcal{E}(N)\log N}. \end{equation}
\end{Theorem}
The formula \eqref{PNTerror} has already been used for a related problem concerning averages of Goldbach representations 
\cite[Eq. (7)]{BhowRuzsa2018}, and similar formulas are well known \cite[Eq. (5.20)]{Kouk2019}.

For our first application, we assume the Riemann Hypothesis and use \eqref{Cramer} and \eqref{SV} to recover \eqref{LZ}, and also obtain the classical Riemann Hypothesis error in the prime number theorem
due to von Koch in 1901 \cite{vonKoch1901}.
\begin{Theorem} \label{thm4}Assuming the Riemann Hypothesis, then $\mathcal{E}(N) \ll N\log^3\!N$ and $\psi(N) = N +O(N^{1/2}\log^2\!N)$.
\end{Theorem} 
For our second application we will strengthen the results in Theorem \ref{thm4} by assuming conjectures on bounds for $J(x,h)$ which we will prove to be consequences of conjectured bounds related to Montgomery's pair correlation conjecture. Montgomery introduced the function, for $x>0$ and $T\ge 3$,
\begin{equation} \label{Fdef} F(x,T) := \sum_{0<\gamma,\gamma'\le T} x^{i(\gamma -\gamma')} w(\gamma -\gamma') , \qquad w(u) = \frac{4}{4+u^2},\end{equation} 
where the sum is over the imaginary parts $\gamma$ and $\gamma'$ of zeta-function zeros. By \cite[Lemma 8]{GoldMont} $F(x,T)$ is real, $F(x,T)\ge 0$, $F(x,T)=F(1/x,T)$, and assuming the Riemann Hypothesis 
\begin{equation} \label{MonThm} F(x,T) = \frac{T}{2\pi} \left( x^{-2}\log^2 T + \log x\right) \left( 1 + O\left(\sqrt{\frac{\log \log T}{\log T}}\right)\right)\end{equation}
 uniformly for $1\le x\le T$.
For larger $x$, Montgomery \cite{Montgomery72} conjectured that, for every fixed $A>1$, 
\begin{equation} \label{Fconj} F(x,T) \sim \frac{T}{2\pi}\log T \end{equation}
holds uniformly for $T\le x \le T^A$. This conjecture implies the pair correlation conjecture for zeros of the zeta-function \cite{Gold2005}.
For all $x>0$ we have the trivial (and unconditional) estimate
\begin{equation} \label{Ftrivial} F(x,T)\le F(1,T) \ll T\log^2 T \end{equation}
which follows from \eqref{N(T+1)-N(T)}.

The main result in \cite{GoldMont} is the following theorem connecting $F(x,T)$ and $J(x,h)$.
\begin{theorem*}[Goldston-Montgomery] Assume the Riemann Hypothesis. Then the $F(x,T)$ conjecture \eqref{Fconj} is equivalent to the conjecture that for every fixed $\epsilon>0$, 
\begin{equation} \label{GM} J(x,h) \sim h x \log(\frac{x}{h}) \end{equation}
holds uniformly for $1\le h\le x^{1-\epsilon} $.
\end{theorem*}

Adapting the proof of this last theorem we obtain the following results.
\begin{Theorem} \label{thm5} Assume the Riemann Hypothesis. 

{\bf (A)} If for any $A>1$ and $T\ge 2$ we have
\begin{equation} \label{A-Fbound} F(x,T) = o( T\log^2 T) \qquad \text{ uniformly for} \quad T\le x \le T^{A}, \end{equation}
then this implies, for $x\ge 2$, 
\begin{equation}\label{A-Jbound} J(x,h) = o(h x \log^2\!x),\qquad \text{for} \quad 1\le h\le x,\end{equation}
and this bound implies $\mathcal{E}(N) = o(N\log^3\!N)$ and $\psi(N) = N +o(N^{1/2}\log^2\!N)$.

{\bf (B)} If for $T\ge 2$
\begin{equation} \label{B-Fbound} F(x,T) \ll T\log x \qquad \text{holds uniformly for} \quad T\le x \le T^{\log T}, \end{equation}
then we have, for $x\ge 2$, 
\begin{equation}\label{B-Jbound} J(x,h) \ll h x \log x,\qquad \text{for} \quad 1\le h\le x,\end{equation}
and this bound implies $\mathcal{E}(N) \ll N\log^2\!N$ and $\psi(N) = N +O(N^{1/2}(\log N)^{3/2})$.
\end{Theorem}

Montgomery's $F(x,T)$ Conjecture \eqref{Fconj} immediately implies \eqref{A-Fbound} so we are using a weaker conjecture on $F(x,T)$ in \Cref{thm5} (A). 
For \Cref{thm5} (B), Montgomery's $F(x,T)$ Conjecture \eqref{Fconj} only implies \eqref{B-Fbound} for $T\le x \le T^A$, and this is a new conjecture for the range $T^A\le x \le T^{\log T}$.
\Cref{thm2} and \Cref{thm5} show that with either assumption, we obtain an improvement on the bound $E(N)\ll N\log^3\!N$ in \eqref{LZ}, which we state as the following corollary.

\begin{corollary} Assume the Riemann Hypothesis. If either \eqref{A-Fbound} or \eqref{A-Jbound} is true, then we have
\begin{equation} \label{PCresult-1} \sum_{n\le N} \psi_2(n) = \frac{ N^2}{2} -2 \sum_{\rho} \frac{N^{\rho+1} }{\rho(\rho +1)}+o( N\log^3\!N),\end{equation} 
while if either \eqref{B-Fbound} or \eqref{B-Jbound} is true, we have
\begin{equation} \label{PCresult-2} \sum_{n\le N} \psi_2(n) = \frac{ N^2}{2} -2 \sum_{\rho} \frac{N^{\rho+1} }{\rho(\rho +1)}+O( N\log^2\!N).\end{equation}
\end{corollary}

We will prove a more general form of the implication that a bound on $F(x,T)$ implies a corresponding bound on $J(x,h)$, but \Cref{thm5} contains the most interesting special cases. We make crucial use of the Riemann Hypothesis bound \eqref{SV} for $h$ very close to $x$. The conjectures \eqref{B-Fbound} and \eqref{B-Jbound} are weaker than what the Goldston-Montgomery theorem suggests are true, but they suffice in Theorem \ref{thm5} and the use of stronger bounds does not improve the results on $\mathcal{E}(N)$. The result on the prime number theorem in Theorem \ref{thm5} (A) is due to Heath-Brown \cite[Theorem 1]{Heath-Brown}. The bound in \eqref{B-Fbound} is trivially true by \eqref{Ftrivial} if $x\ge T^{\log T}$.

For our third application, we consider the situation where there can be zeros of the Riemann zeta-function off the half-line, but these zeros satisfy the bound $1/2<\Theta < 1$, where
\[ \Theta := \sup\{ \beta : \zeta(\beta +i \gamma) =0\}. \] 
The following is a special case of a more general result recently obtained in \cite{Bhowmik-H-M-SGoldbach2019}.
\begin{Theorem}[Bhowmik, Halupczok, Matsumoto, Suzuki] \label{thm6} Assuming $1/2<\Theta < 1$. For $N\ge 2$ and $1\le h\le N$, we have
\begin{equation} J(N,h) \ll h N^{2\Theta}\log^4\!N \qquad \text{and}\quad \mathcal{E}(N) \ll N^{2\Theta}\log^5\!N.\end{equation}
\end{Theorem}

Weaker results of this type were first obtained by Granville \cite{Gran1,Gran2}. To prove this theorem we only need to adjust a few details of the proof in \cite{Bhowmik-H-M-SGoldbach2019} to match our earlier theorems. With more work the power of $\log N$ can be improved but that makes no difference in applications. We will not deal with the situation when $\Theta =1$, which depends on the width of the zero-free region to the left of the line $\sigma =1$ and the unconditional error in the prime number theorem. The converse question of using a bound for $E(N)$ to obtain a zero-free region has been solved in an interesting recent paper of Bhowmik and Ruzsa \cite{BhowRuzsa2018}. As a corollary, we see that the terms in Fujii's formula down to size $\beta \ge 1/2$ are main terms assuming the conjecture $\Theta < 3/4$ instead of the Riemann Hypothesis.
\begin{corollary} \label{cor_BHMS} Suppose $1/2 \le \Theta < 3/4$. Then 
\begin{equation} \label{noRH} \sum_{n\le N} \psi_2(n) = \frac{ N^2}{2} -2 \sum_{\substack{\rho\\ \beta\ge 1/2}} \frac{N^{\rho+1}}{\rho(\rho+1)}+o( N^{3/2}).\end{equation} 
\end{corollary}

The term $E(N)$ in Fujii's formula will give main terms $\gg N^{3/2}$ from zeros with $\beta \ge 3/4$. For weighted versions of Fujii's theorem there are formulas where the error term corresponding to $E(N)$ is explicitly given in terms of sums over zeros, see \cite{Lang-Zac2015} and \cite{BKP}. In principle one could use an explicit formula for $\Psi(z)$ in Theorem \ref{thm1} to obtain a complicated formula for $E(N)$ in terms of zeros, but we have not pursued this.


We conclude with a slight variation on Fujii's formula. We have been counting Goldbach representations using the von Mangoldt function $\Lambda(n)$, so that we include prime powers and also representations for odd integers. Doing this leads to nice formulas such as Fujii's formula because of the weighting by $\log p$ and also that complicated lower order terms coming from prime powers have all been combined into the sum over $\rho$ term in Fujii's formula. The reason for this can be seen in Landau's formula \cite{Landau1912}, which states that for fixed $x$ and $T\to \infty$, 
\begin{equation} \sum_{0<\gamma \le T} x^{\rho} =- \frac{T\Lambda(x)}{2\pi} + O(\log T),\end{equation}
where we define $\Lambda(x)$ to be zero for real non-integer $x$. 
In the following easily proven theorem we remove the Goldbach representations counted by the von Mangoldt function for odd numbers. 
\begin{Theorem} \label{thm7} We have 
\begin{equation} \label{odd}\sum_{\substack{n\le N \\ n \ {\rm odd}}} \psi_2(n) = 2 N \log N +O(N),\end{equation}
and therefore by \eqref{specialcase}
\begin{equation} \label{even}\sum_{\substack{n\le N \\ n \ {\rm even}}} \psi_2(n) = \frac{N^2}{2} - 2\sum_\rho \frac{N^{\rho +1}}{\rho(\rho +1)}
 - 2 N \log N + E(N) +O(N).\end{equation}
\end{Theorem}
The interesting aspect of \eqref{even} is that a new main term $-2N\log N$ has been introduced into Fujii's formula, and this term comes from not allowing representations where the von Mangoldt function is evaluated at the prime 2 and its powers. If we denote the error term $E_{\rm even}(N) := - 2N\log N+E(N)$ in \eqref{even}, then we see that at least one or both of $E(N)$ or $E_{\rm even}(N)$ is $\Omega(N\log N)$. The simplest answer for which possibility occurs would be that the error term in Fujii's formula is smaller than any term generated by altering the support of the von Mangoldt function, in which case $E_{\rm even}(N) = - 2N\log N +o(N\log N)$ and $E(N) = o(N\log N)$. Whether this is true or not seems to be a very difficult question. 

\section{Proofs of Theorems \ref{thm1}, \ref{thm2}, \ref{thm3}, and \ref{thm4}}

\begin{proof}[Proof of Theorem \ref{thm1}]
We first obtain a weighted version of Fujii's theorem.
By \eqref{PNT} we see $\Lambda(n)$ is on average 1 over long intervals, 
and therefore as a first order average approximation to $\Psi(z)$ we use
\begin{equation} I(z) := \sum_{n} z^n = \frac{z}{1-z} \end{equation}
for $|z|<1$. Observe that, on letting $n=m+m'$ in the calculations below, 
\[ I(z) \sum_{m} a_m z^m = \sum_{m,m'} a_m z^{m+m'} =\sum_{n\ge 2} \left(\sum_{ m\le n-1}a_m\right) z^n, \]
and therefore
\[ I(z)^2 = I(z) I(z) 
= \sum_{n\ge 2} \left(\sum_{ m\le n-1}1\right) z^n = \sum_{n} (n-1)z^n ,\]
and
\[ I(z)\Psi(z) =\sum_{n\ge 2}\left(\sum_{ m\le n-1}\Lambda(m)\right) z^n = \sum_n\psi(n-1)z^n.\]
Thus
\[ (\Psi(z)-I(z))^2 = \Psi(z)^2 - 2\Psi(z)I(z) + I(z)^2 = \sum_n \left( \psi_2(n) -2\psi(n-1) + (n-1)\right)z^n ,\]
and we conclude that
\begin{equation} \label{WeightedFujii} \sum_n \psi_2(n) z^n = \sum_n \left(2\psi(n-1) -(n-1)\right)z^n +(\Psi(z)-I(z))^2 ,
\end{equation}
which is a weighted version of \Cref{thm1}.

To remove the weighting, we take $ z=re(\alpha)$ with $0 < r < 1$, and recalling \eqref{I_N} 
we have
\begin{equation} \label{truncate} \int_0^1 (\sum_{n} a_n r^ne(\alpha n)) I_N(1/r,-\alpha) \, d\alpha = \sum_{n}\sum_{n'\le N} a_n r^{n - n'}\int_0^1 e((n-n')\alpha)\, d\alpha = \sum_{n\le N} a_n. \end{equation}
Thus, using \eqref{WeightedFujii} with $z=re(\alpha)$ and $0 < r < 1$ in \eqref{truncate}, we obtain
\begin{align*}
\sum_{n\le N} \psi_2(n) &= \sum_{n\le N}(2\psi(n-1)-(n-1)) + \int_0^1(\Psi(r,\alpha)-I(r,\alpha))^2 I_N(1/r,-\alpha)\, d\alpha \\
&= \sum_{n\le N}(2\psi(n-1)-(n-1)) + E(N).
\end{align*}
Utilizing \cite[Chapter 2, (13)]{Ingham1932}
$\psi_1(x) := \int_0^x \psi(t)\, dt$ and $\psi_1(N) = \sum_{n\le N}\psi(n-1)$,
we have
\begin{equation} \label{M(N)term} \sum_{n\le N} \psi_2(n) = 2\psi_1(N) - \frac12 (N-1)N + E(N). \end{equation}
To complete the proof, we use the explicit formula, for $x\ge 1$,
\begin{equation}\label{Psi1} \psi_1(x) = \frac{x^2}{2} -\sum_\rho \frac{x^{\rho +1}}{\rho(\rho +1)} - (\log 2\pi)x +\frac{\zeta '}{\zeta}(-1) - \sum_k \frac{x^{1-2k}}{2k(2k-1)}, \end{equation}
see \cite[Theorem 28]{Ingham1932} or \cite[12.1.1 Exercise 6]{MontgomeryVaughan2007}.
Substituting \eqref{Psi1} into \eqref{M(N)term}, we obtain \Cref{thm1}.
\end{proof}

\begin{proof}[Proof of Theorem \ref{thm2}]
Letting
\begin{equation} \label{LambdaNot} \Lambda_0(n) := \Lambda(n) - \1_{n\ge 1}, \end{equation}
we have
\begin{align*} \mathcal{E}(N) &:= \int_0^1 |\Psi(r,\alpha)-I(r,\alpha)|^2 |I_N(1/r,-\alpha)|\, d\alpha \\
&= 2\int_0^{1/2} \left| \sum_n \Lambda_0(n)r^n e(n\alpha) \right|^2 |I_N(1/r,-\alpha)| \,d\alpha. \end{align*}
We will now choose $r$ in terms of $N$ by setting
\begin{equation} \label{rN} r=e^{-1/N}, \end{equation}
and with this choice it is easy to see that for $|\alpha| \le 1/2$ we have
\begin{equation} \label{INbound} |I_N(1/r,-\alpha)|
= \left| \frac{e-e(\alpha N)}{r - e(\alpha)}\right| 
\ll \min(N, \frac1{|\alpha|}).\end{equation}
Therefore we obtain
\begin{equation} \label{ENbound} \mathcal{E}(N) \ll \int_0^{1/2} \left| \sum_n \Lambda_0(n)r^n e(n\alpha) \right|^2 \min(N, \frac1{\alpha}) \, d\alpha. \end{equation}

Letting 
\[ h_N(\alpha) := N \1_{0\le \alpha \le \frac2N } +\sum_{k<\log_2\!N} \frac{N}{2^k} \1_{\frac{2^k}N < \alpha \le \frac{2^{k+1}}{N}}, \]
 then $ \frac12 h_N(\alpha) \le \min(N, \frac1{\alpha}) \le h_N(\alpha)$ in the range $0\le \alpha \le 1/2$.
Now if we put
\[ H_N(\alpha) := \sum_{0\le k<\log_2\!N} \frac{N}{2^k} \1_{0\le \alpha \le \frac{2^{k+1}}{N }}, \]
then $h_N(\alpha)\le H_N(\alpha) \le 2 h_N(\alpha)$ and therefore $ \frac14 H_N(\alpha) \le \min(N, \frac1{\alpha}) \le H_N(\alpha) $
in the range $0\le \alpha \le 1/2$. We conclude 
\[\min(N, \frac1{\alpha}) \asymp h_N(\alpha) \asymp H_N(\alpha), \qquad \text{for} \quad 0\le \alpha \le 1/2. \]
Thus
\begin{equation} \label{Breakup} \mathcal{E}(N) \ll \sum_{0\le k< \log_2\!N} \frac{N}{2^k} \int_0^{2^{k+1}/N} \left| \sum_n \Lambda_0(n) r^n e(n\alpha) \right|^2 \, d\alpha = \sum_{0\le k< \log_2\!N} \frac{N}{2^k} \mathcal{W}(\frac{N}{2^{k+2}}).
\end{equation}

Gallagher's lemma gives the bound (see \cite[Lemma 1.9]{Montgomery71} or \cite[Section 4]{Gold-Yang})
 \[ \int_0^{1/2h}\left|\sum_n a_n e(n\alpha) \right|^2 \, d\alpha \ll \frac1{h^2} \int_{-\infty}^\infty \left| \sum_{ x<n\le x+h} a_n \right|^2 \, dx\]
provided $\sum_n|a_n| <\infty$, 
and therefore
we have 
\begin{align*} \mathcal{W}(N,h) &= \int_0^{1/2h} \left| \sum_n \Lambda_0(n)r^n e(n\alpha) \right|^2 \, d\alpha \\
&\ll \frac1{h^2} \int_{-\infty}^\infty \left| \sum_{ x<n\le x+h} \Lambda_0(n)r^n \right|^2 \, dx \\
&= \frac1{h^2} \int_{-h}^0 \left| \sum_{ n\le x+h} \Lambda_0(n)r^n \right|^2 \, dx + \frac1{h^2} \int_{0}^\infty \left| \sum_{ x<n\le x+h} \Lambda_0(n)r^n \right|^2 \, dx .
\end{align*}
We conclude
\begin{equation} \label{W(h)} \mathcal{W}(N,h) \ll \frac1{h^2} \left( I_1(N,h) +I_2(N,h)\right), \end{equation}
where
\begin{equation} \label{I1 and I2} I_1(N,h) := \int_0^h \left| \sum_{ n\le x} \Lambda_0(n)e^{-n/N} \right|^2 \, dx, \qquad I_2(N,h) := \int_0^\infty \left|\sum_{x< n\le x+h} \Lambda_0(n)e^{-n/N} \right|^2 \, dx . \end{equation}

To bound $I_1(N,h)$ and $I_2(N,h)$ we will use partial summation on the integrands with the counting function
\begin{equation} \label{R(u)def}
R(u) := \sum_{n\le u}\Lambda_0(n) = \psi(u) - \lfloor u \rfloor .\end{equation}
Recalling $H(x)$ and $J(x,h)$ from \eqref{HandJ}, and making use here and later of the inequality $(a+b)^2\le 2(a^2+b^2)$, we have, for $x\ge1$,
\begin{equation} \begin{split} \label{R-Hbound}
\int_0^x R(t)^2\, dt &= \int_0^x \left( \psi(t) - \lfloor t \rfloor \right)^2\, dt
= \int_0^x \left( \psi(t)-t + t - \lfloor t \rfloor \right)^2\, dt \\
&\ll \int_0^x \left( \psi(t)-t \right)^2\, dt + \int_0^x \left( t - \lfloor t \rfloor \right)^2\, dt \\
&\ll H(x) + x, \end{split} \end{equation}
and 
\begin{equation} \begin{split} \label{R-Jbound} \int_0^x \left(R(t+h)-R(t)\right)^2\, dt
&\ll J(x,h) + \int_0^x (\lfloor t+h\rfloor - \lfloor t\rfloor -h)^2 dt \\ &
\ll J(x,h) + x.
\end{split}\end{equation}
Next, by partial summation, 
\begin{equation} \label{weightsum} \sum_{n\le x}\Lambda_0(n) e^{-n/N} = \int_0^x e^{-u/N} dR(u) 
= R(x)e^{-x/N} + \frac1N\int_0^xR(u) e^{-u/N} \, du,
\end{equation}
and therefore, for $1\le h \le N$, we have, using Cauchy-Schwarz and \eqref{R-Hbound},
\begin{equation}\label{I_1bound} \begin{split} I_1(N,h) & \ll \int_0^h\left( R(x)^2e^{-2x/N} + \frac1{N^2}\left(\int_0^x|R(u)| e^{-u/N} \, du\right)^2 \right)\, dx \\& \ll \int_0^h\left( R(x)^2e^{-2x/N} + \frac1{N^2}\left(\int_0^xR(u)^2 e^{-u/N} \, du\right) \left(\int_0^x e^{-u/N}\,du\right)\right)\, dx \\&
\ll \left(1+ \frac{h^2}{N^2}\right) \int_0^h R(x)^2\, dx \ll \int_0^h R(x)^2\, dx\ll H(h) + h. \end{split} \end{equation}

For $I_2(N,h)$, we replace $x$ by $x+h$ in \eqref{weightsum} and take their difference to obtain
\begin{equation}\begin{split}\label{differencesum} \sum_{x<n\le x+h}\Lambda_0(n) e^{-n/N} &= 
R(x+h)e^{-(x+h)/N} - R(x)e^{-x/N} + \frac1N\int_x^{x+h}R(u) e^{-u/N} \, du \\
&= \left(R(x+h)-R(x)\right) e^{-x/N} + O\left( \frac{h}{N}|R(x+h)|e^{-x/N} \right) \\
&\qquad+O\left(\frac1N\int_x^{x+h}|R(u)|e^{-u/N}\, du\right).
\end{split} \end{equation}
Thus we have
\begin{align*} I_2(N,h) &\ll \int_0^\infty (R(x+h)-R(x))^2 e^{-2x/N}\, dx + \frac{h^2}{N^2} \int_0^\infty R(x+h)^2 e^{-2x/N}\, dx \\
&\qquad+ \frac1{N^2} \int_0^\infty \left(\int_x^{x+h} |R(u)|e^{-u/N}\, du \right)^2\, dx. \end{align*}
Applying the Cauchy-Schwarz inequality to the double integral above and changing the order of integration, we bound this term by
\[ \le \frac{h}{N^2} \int_0^\infty \int_x^{x+h} R(u)^2e^{-2u/N}\, du \, dx = \frac{h^2}{N^2} \int_0^\infty R(u)^2e^{-2u/N}\, du \]
and therefore, for $1\le h \le N$,
\[ I_2(N,h) \ll \frac{h^2}{N^2} \int_0^\infty R(x)^2 e^{-2x/N}\, dx
+\int_0^\infty (R(x+h)-R(x))^2 e^{-2x/N}\, dx.\]
Now for any integrable function $f(x)\ge 0$ we have
\begin{align*} \int_0^\infty f(x)e^{-2x/N}\, dx &\le \int_0^N f(x)\, dx + \sum_{j=1}^\infty e^{-2j} \int_{jN}^{(j+1)N} f(x)\, dx \\
&\le \sum_{j=1}^\infty \frac{1}{2^{j-1}} \int_0^{jN} f(x)\, dx, \end{align*} 
and therefore 

\begin{equation}\begin{split} \label{I_2bound} I_2(N,h) &\ll \sum_{j=1}^\infty\frac{1}{2^{j}} \int_0^{jN}\left( \frac{h^2}{N^2} R(x)^2+ (R(x+h)-R(x))^2 \right)\, dx \\& \ll \sum_{j=1}^\infty\frac{1}{2^{j}}\left( \frac{h^2}{N^2}H(jN)+ \frac{jh^2}{N} +J(jN,h) + jN \right) \\& 
\ll \sum_{j=1}^\infty\frac{1}{2^{j}}\left(\frac{h^2}{N^2}H(jN)+ J(jN,h) \right) +N
 .\end{split}\end{equation}
 Thus by \eqref{W(h)}, \eqref{I_1bound} and \eqref{I_2bound}, we conclude, for $1\le h\le N$

\begin{equation}
\mathcal{W}(N,h) \ll \frac1{h^2} H(h) + \frac{1}{N^2}\sum_{j=1}^\infty \frac{1}{2^{j}}H(jN) + \frac1{h^2}\sum_{j=1}^\infty \frac{1}{2^{j}}J(jN,h) + \frac{N}{h^2}. 
\end{equation}
\end{proof}

\begin{proof}[Proof of Theorem \ref{thm3}] From \eqref{LambdaNot} we have 
$ \Psi(z) - I(z) = \sum_n \Lambda_0(n)z^n$,
and \eqref{PNTerror} follows from \eqref{truncate}. To obtain \eqref{PNTbound}, applying the Cauchy-Schwarz inequality to \eqref{PNTerror} and using \eqref{INbound} we have
\[ \psi(N) - N \ll \sqrt{\mathcal{E}(N) \int_0^{1/2}\min(N,\frac1{\alpha})\, d\alpha } \ll \sqrt{\mathcal{E}(N) \log N}.\]
\end{proof}

\begin{proof}[Proof of Theorem \ref{thm4}] Assuming the Riemann Hypothesis, from \eqref{Cramer} and \eqref{SV} we have $H(x)\ll x^2$ and $J(x,h) \ll hx\log^2\!x$ for $1\le h \le x$, and therefore by \eqref{Wbound} we have $\mathcal{W}(N,h)\ll \frac{N}{h}\log^2\!N $ for $1\le h\le N$, and therefore $\mathcal{E}(N) \ll N\log^3\!N$, which is the first bound in Theorem \ref{thm4}, and Theorem \ref{thm3} gives the second bound.
\end{proof}

\section{Theorem \ref{thm8} and proof of Theorem \ref{thm5}}

We will prove the following more general form of Theorem \ref{thm5}. In addition to $J(x,h)$ from \eqref{HandJ}, we also make use of the related variance 
\begin{equation} \label{calJ} \mathcal{J}(x,\delta)
 := \int_0^{x} \left( \psi((1+\delta)t) - \psi(t) - \delta t \right)^2\, dt, \qquad \text{ for} \quad 0<\delta \le 1. \end{equation}
The variables $h$ in $J(x,h)$ and $\delta$ in $\mathcal{J}(x,\delta)$ are roughly related by $h\asymp \delta x$. Saffari and Vaughan \cite{SaffariVaughan} proved in addition to \eqref{SV} that assuming the Riemann Hypothesis, we have, for $0< \delta \le 1$,
\begin{equation} \label{SV2} \mathcal{J}(x,\delta) \ll \delta x^2 \log^2(2/\delta). \end{equation}
\begin{Theorem} \label{thm8} Assume the Riemann Hypothesis. Let $x\ge 2$, and $\mathcal{L}(x)$ be a continuous increasing function satisfying 
\begin{equation} \label{calL-bound} \log x \le \mathcal{L}(x) \le \log^2\!x. \end{equation} Then the assertion
\begin{equation} \label{Fthm8} F(x,T) \ll T\mathcal{L}(x) \qquad \text{ uniformly for } \quad e^{\sqrt{\mathcal{L}(x)}} \le T \le x , \end{equation}
implies the assertion
\begin{equation} \label{calJthm8} \mathcal{J}(x,\delta) \ll \delta x^2\mathcal{L}(x) \qquad \text{ uniformly for } \quad 1/x \le \delta \le 2e^{-\sqrt{\mathcal{L}(x)}} , \end{equation}
which implies the assertion
\begin{equation} \label{Jthm8} J(x,h) \ll hx\mathcal{L}(x) \qquad \text{ uniformly for } \quad 1 \le h \le 2xe^{-\sqrt{\mathcal{L}(x)}}, \end{equation} 
which implies $\mathcal{E}(N)\ll N\mathcal{L}(N) \log N$ and $\psi(N) = N +O(N^{1/2}\sqrt{\mathcal{L(N)}}\log N)$.
\end{Theorem}

Montgomery's function $F(x,T)$ is used for applications to both zeros and primes. For applications to zeros, it is natural to first fix a large $T$ and consider zeros up to height $T$ and then pick $x$ to be a function of $T$ that varies in some range depending on $T$. This is how Montgomery's conjecture is stated in \eqref{Fconj} and also how the conjectures \eqref{A-Fbound} and \eqref{B-Fbound} in \Cref{thm5} are stated. However, in applications to primes, following Heath-Brown \cite{Heath-Brown}, it is more convenient to fix a large $x$ and consider the primes up to $x$, and then pick $T$ as a function of $x$ that varies in some range depending on $x$. This is what we have done in \Cref{thm8}. 

The ranges we have used in our conjectures on $F(x,T)$, $\mathcal{J}(x,\delta)$ and $J(x,T)$ in \Cref{thm8} are where these conjectures are needed. In proving \Cref{thm8}, however, it is convenient to extend these ranges to include where the bounds are known to be true on the Riemann Hypothesis. 
\begin{lem} \label{lem1} Assume the Riemann Hypothesis. Letting $x\ge 2$, then the assertion \eqref{Fthm8} implies, for any bounded $A\ge 1$,
\begin{equation} \label{Flem1} F(x,T) \ll T\mathcal{L}(x) \qquad \text{ uniformly for } \quad 2 \le T \le x^A ; \end{equation}
the assertion \eqref{calJthm8} implies 
\begin{equation} \label{calJlem1} \mathcal{J}(x,\delta) \ll \delta x^2\mathcal{L}(x) \qquad \text{ uniformly for } \quad 0 \le \delta \le 1; \end{equation}
and the assertion \eqref{Jthm8} implies 
\begin{equation} \label{Jlem1} J(x,h) \ll hx\mathcal{L}(x) \qquad \text{ uniformly for } \quad 0 \le h \le x, \end{equation} 
\end{lem}

\begin{proof}[Proof of \Cref{lem1}]

We first prove \eqref{Flem1}. By the trivial bound \eqref{Ftrivial}, we have unconditionally in the range $2\le T\le e^{\sqrt{\mathcal{L}(x)}}$ that
\[ F(x, T)\ll T \log^2T\ll T\mathcal{L}(x), \]
which with \eqref{Fthm8} proves \eqref{Flem1} for the range $2\le T\le x$. For the range $x\le T \le x^A$ we have $T^{1/A}\le x \le T$, and therefore on the Riemann Hypothesis \eqref{MonThm} implies $F(x,T) \sim T\log x \ll T\mathcal{L}(x) $ by \eqref{calL-bound}.

Next, for \eqref{calJlem1} in the range $ 2e^{-\sqrt{\mathcal{L}(x)}}\le \delta \le 1$ we have 
$\log^2(2/\delta) \le \mathcal{L}(x) $ and therefore on the Riemann Hypothesis by \eqref{SV2} we have $\mathcal{J}(x,\delta) \ll \delta x^2 \log^2(2/\delta)\ll \delta x^2 \mathcal{L}(x) $ in this range. For the range $0\le \delta < 1/x$ we have $0\le \delta x < 1$ and therefore there is at most one integer in the interval $(t, t+\delta t]$ for $0\le t \le x$. Hence
\begin{equation} \begin{split}\label{calJtrivial} \mathcal{J}(x, \delta) & \ll \int_0^{x} \left( \psi((1+\delta)t) - \psi(t) \right)^2\, dt + \int_0^{x} (\delta t )^2\, dt \\&
\ll \sum_{n\le x+1}\Lambda(n)^2\int_{n/(1+\delta)}^n \, dt +\delta^2x^3 \\&
\ll \delta \sum_{n\le x+1}\Lambda(n)^2 n +\delta^2x^3 \\&
\ll \delta x^2\log x + \delta^2x^3
\\& \ll \delta x^2( \log x + \delta x) \ll \delta x^2\log x \ll \delta x^2\mathcal{L}(x) .
\end{split} \end{equation}
By this and \eqref{calJthm8} we obtain \eqref{calJlem1}. A nearly identical proof for $J(x,h)$ shows that \eqref{Jthm8} implies \eqref{Jlem1}. 
\end{proof}

\begin{proof}[Proof of Theorem \ref{thm5} from Theorem \ref{thm8}] To prove (A), choose $\mathcal{L}(x) = \epsilon \log^2\!x$. The assumption \eqref{Fthm8} becomes $F(x,T) \ll \epsilon T\log^2\!x$ for $x^{\sqrt{\epsilon}} \le T \le x$, or equivalently $T \le x \le T^{1/\sqrt{\epsilon}}$. Letting $A=1/\sqrt{\epsilon}$, we have $F(x,T) \ll A^{-2} T\log^2\!x$ for $T \le x \le T^A$. The assertion \eqref{A-Fbound} implies that this bound for $F(x,T)$ holds since 
\[ F(x,T) = o(T\log^2T) \ll o(T\log^2\!x) \ll A^{-2} T\log^2\!x\]
if $A\to \infty$ sufficiently slowly. Thus \eqref{Jthm8} holds which implies by \Cref{lem1}, $J(x,h) \ll \epsilon h x \log^2\!x$ for $1\le h\le x$. Also by \Cref{thm8} $\mathcal{E}(N)\ll \epsilon N \log^3\!N$ and $\psi(N) = N +O(\sqrt{\epsilon}N^{1/2}\log^2\!N)$, where $\epsilon$ can be taken as small as we wish.

To prove \Cref{thm5} (B), choose $\mathcal{L}(x) = \log x$. The assumption \eqref{Fthm8} becomes $F(x,T) \ll T\log x$ for $e^{\sqrt{\log x}} \le T \le x$, or equivalently $T \le x \le T^{\log T}$ and this is satisfied when \eqref{B-Fbound} holds. Thus by \eqref{Jthm8} and \Cref{lem1} $J(x,h) \ll h x \log x$ for $1\le h\le x$, and by \Cref{thm8} $\mathcal{E}(N)\ll N \log^2\!N$ and $\psi(N) = N +O(N^{1/2}\log^{3/2} N)$.
\end{proof}

\begin{proof}[Proof of \Cref{thm8}, \eqref{Fthm8} implies \eqref{calJthm8}]
We start from the easily verified identity
\[ F(x,T) = \frac2{\pi}\int_{-\infty}^\infty \left|\sum_{0<\gamma\le T}\frac{x^{i\gamma}}{1+(t-\gamma)^2}\right|^2\, dt, \]
which is implicitly in \cite{Montgomery72}, see also \cite[Eq. 26]{GoldMont} and \cite[Section 4]{Gold2005}. Next, Montgomery showed, using \eqref{N(T+1)-N(T)},
\begin{equation} \label{F-integral} F(x,T) = \frac2{\pi}\int_0^T \left|\sum_\gamma\frac{x^{i\gamma}}{1+(t-\gamma)^2}\right|^2\, dt +O(\log^3 T) . \end{equation}

The main tool in proving this part of \Cref{thm8} is the following result.
\begin{lem} \label{lem2} 
For $0<\delta\le 1$ and $e^{2\kappa} = 1+\delta$, let
\begin{equation} \label{G} G(x,\delta):= \int_0^\infty \left(\frac{\sin \kappa t}{t}\right)^2 \left|\sum_\gamma\frac{x^{i\gamma}}{1+(t-\gamma)^2}\right|^2\, dt. \end{equation}
Assuming the Riemann Hypothesis, we have, for $1/x \le \delta \le 1$,
\begin{equation} \label{calJlem2} \mathcal{J}(x,\delta) \ll x^2 G(x,\delta) + O(\delta x^2).
\end{equation}
\end{lem}
\noindent We will prove \Cref{lem2} after completing the proof of \Cref{thm8}. 

Since $\kappa = \frac12\log(1+\delta) \asymp \delta$ for $0<\delta \le 1$,
we have 
\[ 0\le \left(\frac{\sin \kappa t}{t}\right)^2 \ll \min(\kappa^2 , 1/t^2)\ll \min( \delta^2, 1/t^2),\]
and hence
\[ G(x,\delta) 
\ll \int_0^\infty \min\left( \delta^2, \frac{1}{t^2}\right) \left|\sum_\gamma\frac{x^{i\gamma}}{1+(t-\gamma)^2}\right|^2\, dt .
\]
For $U\ge 1/\delta$ we have, since by \eqref{N(T+1)-N(T)} $\sum_\gamma 1/(1+(t-\gamma)^2)\ll \log (|t|+2)$,
\[ \int_U^\infty \frac{1}{t^2} \left|\sum_\gamma\frac{x^{i\gamma}}{1+(t-\gamma)^2}\right|^2\, dt \ll \int_U^\infty \frac{\log^2 t}{t^2} \, dt \ll \frac{\log^2U}{U}. \]
On taking $U=2\log^2(2/\delta)/\delta$, we have by \eqref{F-integral}
\[ \begin{split} G(x,\delta) &\ll \int_0^U \min\left( \delta^2, \frac{1}{t^2}\right) \left|\sum_\gamma\frac{x^{i\gamma}}{1+(t-\gamma)^2}\right|^2\, dt +O(\delta)\\&
\ll \delta^2\int_0^{2/\delta}\left|\sum_\gamma\frac{x^{i\gamma}}{1+(t-\gamma)^2}\right|^2\, dt \\&
\qquad+ \sum_{ k\ll \log\log (4/\delta)}\frac{\delta^2}{2^{2k}}\int_{2^k/\delta}^{2^{k+1}/\delta}\left|\sum_\gamma\frac{x^{i\gamma}}{1+(t-\gamma)^2}\right|^2\, dt +O(\delta) \\&
\ll \sum_{ k\ll\log\log(4/\delta)}\frac{\delta^2}{2^{2k}}\left( F(x, 2^k/\delta)+\log^3(2^k/\delta)\right) +O(\delta)\\& \ll \delta^2 \sum_{ k\ll \log\log(4/\delta) }\frac{1}{2^{2k}} F(x, 2^k/\delta)\ + O(\delta).
\end{split}
\]
We now assume \eqref{Fthm8} holds and therefore by \Cref{lem1}, \eqref{Flem1} also holds. Taking
$ 1/x \le \delta \le 1$, we see that for $1\le k \ll \log\log(4/\delta)$ we have $2\le 2^k/\delta\ll x\log^Cx$, for a constant $C$. Therefore $F(x,2^k/\delta)$ is in the range where \eqref{Flem1} applies, and therefore
\[ \begin{split} G(x,\delta) & \ll \delta^2 \sum_{ k\ll \log\log(4/\delta) }\frac{1}{2^{2k}} (2^k/\delta)\mathcal{L}(x) \ +O(\delta) \\ &
\ll \delta \mathcal{L}(x),
\end{split}
\]
which by \eqref{calJlem2} proves \eqref{calJthm8} over a wider range of $\delta$ than required.
\end{proof}

\begin{proof}[Proof of \Cref{thm8}, \eqref{calJthm8} implies \eqref{Jthm8}, and the remaining results]

To complete the proof of \Cref{thm8} we need the following lemma of Saffari-Vaughan \cite[Eq. (6.21)]{SaffariVaughan}, see also \cite[pp. 126--127]{GV1996}, which we will prove later. 
\begin{lem}[Saffari-Vaughan] \label{lem3} For any $1\le h \le x/4$, and any integrable function $f(x)$, we have
\begin{equation} \int_{x/2}^x (f(t+h)-f(t))^2\, dt \le \frac{2x}{h} \int_0^{8h/x}\left(\int_0^x(f(y+\delta y)- f(y) )^2 \, dy\right) \, d\delta. \end{equation}
\end{lem}

Taking $f(t) = \psi(t) - t$ in \Cref{lem3}, and assuming $1\le h \le x/8$ so that we may apply \eqref{calJlem1}, we have
\[ \begin{split} J(x,h) - J(x/2, h) & \ll \frac{x}{h} \int_0^{8h/x} \mathcal{J}(x,\delta) \, d\delta \\
& \ll \frac{x^3\mathcal{L}(x)}{h} \int_0^{8h/x} \delta \, d\delta \\&
\ll h x \mathcal{L}(x). \end{split} \]
Replacing $x$ by $x/2, x/4, \ldots, x/2^{k-1}$ and adding we obtain
\[\begin{split} J(x,h) - J(x/2^k, h) &= \sum_{j\le k}\left(J(x/2^{j-1},h) - J(x/2^j, h)\right) \\& \ll h \sum_{j\le k} (x/2^{j-1}) \mathcal{L}(x/2^{j-1}) \\ & \ll h x \mathcal{L}(x), 
\end{split}\]
where we used $\mathcal{L}(x/2^{j-1})\le \mathcal{L}(x)$, and note that here we need $h\le\frac{x}{2^{k+2}}$.
Taking $k$ so that $\log^2\!x \le 2^{k} \le 2\log^2\!x$, then by \eqref{SV} we have 
\[ J(x/2^{k},h) \le J(x/\log^2\!x,h) \ll hx. \]
Thus
\[ J(x,h) \ll hx\mathcal{L}(x), \]
for $1\le h\le \frac{x}{8\log^2\!x}\le
\frac{x}{2^{k+2}}$.
This proves \eqref{Jthm8} over a larger range of $h$ than required. 

Now applying to \eqref{Wbound} the estimates \eqref{Cramer} and \eqref{Jlem1} (which is implied by \eqref{Jthm8}), we obtain
\[
\mathcal{W}(N,h) \ll \frac{N\mathcal{L}(N)}{h}, \qquad 1\le h\le N,
\]
which by \eqref{E(N)bound} gives the bound $\mathcal{E}(N)\ll N\mathcal{L}(N) \log N$ and this with \eqref{PNTbound} give
\[ \psi(N) = N +O(N^{1/2}\sqrt{\mathcal{L(N)}}\log N). \]

\end{proof}

\begin{proof}[Proof of \Cref{lem2}]

We define $e^{2\kappa} = 1+\delta$, $0<\delta \le 1$, and define
\begin{equation} \label{a(s)} a(s) := \frac{(1+\delta)^s-1}{s}. \end{equation}
Thus 
\[ |a(it)|^2 = 4\left(\frac{\sin \kappa t}{t}\right)^2,\]
and
\[ \begin{split} G(x,\delta) &= \frac14 \int_0 ^\infty |a(it)|^2 \left|\sum_\gamma\frac{x^{i\gamma}}{1+(t-\gamma)^2}\right|^2\, dt \\ &
= \frac18 \int_{-\infty} ^\infty |a(it)|^2 \left|\sum_\gamma\frac{x^{i\gamma}}{1+(t-\gamma)^2}\right|^2\, dt, \end{split} \]
on noting the integrand is an even function since in the sum for every $\gamma$ there is a $-\gamma$.
The next step is to bring $|a(it)|^2$ into the sum over zeros using \cite[Lemma 10]{GoldMont}, from which we immediately obtain, for $Z\ge 1/\delta$,
\[ G(x, \delta) = \frac18\int_{-\infty}^\infty \left|\sum_{|\gamma|\le Z}\frac{a(1/2+i\gamma)x^{i\gamma}}{1+(t-\gamma)^2}\right|^2\, dt +O\left(\delta^2\log^3(2/\delta)\right) +O\left(\frac{\log^3\!Z}{Z}\right).\]
We comment that the proof of Lemma 10 in \cite{GoldMont} is an elementary argument making use of \eqref{N(T+1)-N(T)}. 
We have not made use of the Riemann Hypothesis yet but henceforth we will assume and use $\rho= 1/2+i\gamma$ and $a(\rho) = a(1/2+i\gamma)$. In order to keep our error terms small, we now choose
\begin{equation} \label{Zdelta} Z = x\log^3\!x, \qquad \text{and} \quad 1/x \le \delta \le 1. \end{equation}
Thus 
\begin{equation} G(x, \delta) + O(\delta) = \frac18\int_{-\infty}^\infty \left|\sum_{|\gamma|\le Z}\frac{a(1/2+i\gamma)x^{i\gamma}}{1+(t-\gamma)^2}\right|^2\, dt .\end{equation}
Defining the Fourier transform by
\[ \widehat{f}(u) := \int_{-\infty}^\infty f(t) e(-tu)\, dt. \]
Plancherel's theorem says that if $f(t)$ is in $L^1 \cap L^2$ then $\widehat{f}(u)$ is in $L^2$ and we have
\[ \int_{-\infty}^\infty |f(t)|^2 \, dt = \int_{-\infty}^\infty |\widehat{f}(u)|^2 \, du.\]
An easy calculation gives the Fourier transform pair
\[ g(t) = \sum_{|\gamma|\le Z}\frac{a(\rho)x^{i\gamma}}{1+(t-\gamma)^2}, \qquad \widehat{g}(u) = \pi \sum_{|\gamma|\le Z}a(\rho)x^{i\gamma} e(-\gamma u)e^{-2\pi |u|} , \]
and therefore by Plancherel's theorem we have, with $y= 2\pi u$ in the third line below
\begin{equation}\label{G-Plancherel} \begin{split} G(x,\delta) +O(\delta) &= \frac{\pi^2}{8}\int_{-\infty}^\infty \bigg|\sum_{|\gamma|\le Z} a(\rho)x^{i\gamma}e(-\gamma u)\bigg|^2e^{-4\pi |u|}\, du \\& 
= \frac{\pi^2}{8} \int_{-\infty}^\infty \bigg|\sum_{|\gamma|\le Z} a(\rho)(xe^{-2\pi u})^{i\gamma}\bigg|^2e^{-4\pi |u|}\, du \\& 
=
\frac{\pi}{16} \int_{-\infty}^\infty \bigg|\sum_{|\gamma|\le Z} a(\rho)(xe^{-y})^{i\gamma}\bigg|^2e^{-2 |y|}\, dy \\& \ge \frac{\pi}{16} \int_0^\infty \bigg|\sum_{|\gamma|\le Z} a(\rho)(xe^{-y})^{i\gamma}\bigg|^2e^{-2 y}\, dy . \end{split}
\end{equation}
On letting $t=xe^{-y}$ in the last integral we obtain
\[ \int_0^\infty \bigg|\sum_{|\gamma|\le Z} a(\rho)(xe^{-y})^{i\gamma}\bigg|^2e^{-2 y}\, dy = \frac{1}{x^2} \int_0^{x} \bigg|\sum_{|\gamma|\le Z} a(\rho)t^{\rho}\bigg|^2\, dt ,\]
 and we conclude that 
 \begin{equation} \int_0^{x} \bigg|\sum_{|\gamma|\le Z} a(\rho)t^{\rho}\bigg|^2\, dt \ll x^2 G(x,\delta) + O(\delta x^2). \end{equation}
We now complete the proof of \Cref{lem2} by proving that we have
\begin{equation}\label{Jzeroformula} \int_0^{x} \bigg|\sum_{|\gamma|\le Z} a(\rho)t^{\rho}\bigg|^2\, dt =\mathcal{J}(x,\delta)+O(\delta x^2). \end{equation}

By the standard truncated explicit formula \cite[Chapter 12, Theorem 12.5]{MontgomeryVaughan2007} (also see \cite[Eq. 34]{GoldMont}), we have for $2\le t\le x$, and $Z\ge x$
\begin{equation}\label{explicit} - \sum_{|\gamma|\le Z} a(\rho) t^\rho = \psi((1+\delta)t) - \psi(t) - \delta t + E(t,Z), \end{equation} 
where 
\begin{equation} \label{expliciterror}E(t,Z) \ll \frac{t\log^2(tZ)}{Z} + \log t \min(1,\frac{t}{Z\lVert t\rVert}) + \log t \min(1,\frac{t}{Z\lVert (1+\delta)t\rVert}) ,\end{equation}
and $\lVert u\rVert$ is the distance from $u$ to the nearest integer.
Using the trivial estimate 
\[\psi((1+\delta)t) - \psi(t)\ll \delta t \log t ,\] 
we have
\[ \left|\sum_{|\gamma|\le Z} a(\rho) t^\rho\right|^2 = ( \psi((1+\delta)t) - \psi(t) - \delta t)^2 + O( \delta t (\log t) |E(t,Z)| ) + O(|E(t,Z)|^2).\]
There is a small complication at this point since we want to integrate both sides of this equation over $0\le t\le x$, but \eqref{explicit} requires $2\le t\le x$. Therefore 
integrating from $2\le t \le x$, we obtain
\begin{equation} \label{calJmain} \int_0^{x} \bigg|\sum_{|\gamma|\le Z} a(\rho)t^{\rho}\bigg|^2\, dt = \mathcal{J}(x,\delta) - \mathcal{J}(2,\delta) +\int_0^2\bigg|\sum_{|\gamma|\le Z} a(\rho)t^{\rho}\bigg|^2\, dt + O(E^*), \end{equation}
where 
\begin{equation} \label{Estar} E^*= \delta x \log x\int_2^x|E(t,Z)|\, dt + \int_2^x|E(t,Z)|^2\, dt. \end{equation}
We note first that by \eqref{calJtrivial} 
\[\mathcal{J}(2,\delta) \ll \delta.\]
Next, for $|\text{Re}(s)|\ll 1$ we have $|a(s)|\ll \min(1,1/|s|)$, and by \eqref{N(T+1)-N(T)} we obtain $$ \sum_{|\gamma|\le Z} |a(\rho)|\ll \log^2\!Z. $$
Thus
\[ \int_0^{2} \bigg|\sum_{|\gamma|\le Z} a(\rho)t^{\rho}\bigg|^2\, dt \ll \log^4\!Z \ll \log^4\!x .\]
Both of these errors are negligible compared to the error term $\delta x^2 \gg x$.
It remains to estimate $E^*$. 
First, for $j=1$ or $2$, 
\[ \begin{split} \int_2^x \min(1, \frac{t}{Z\lVert t\rVert})^j \, dt & \ll \sum_{n\le 2x } \int_{n}^{n+1/2} \min(1, \frac{n}{Z(t-n)})^j \, dt \\&
\ll \sum_{n\le 2x } \left(\frac{n}{Z} + (\frac{n}{Z})^j\int_{n+n/Z}^{n+1/2} \frac{1}{(t-n)^j} \, dt\right) \\& 
\ll \frac{x^2\log^{2-j}\!Z}{Z} \ll \frac{x}{\log^{j+1}\!x}. \end{split}\]
The same estimate holds for the term in \eqref{expliciterror} with $\lVert(1+\delta)t\rVert$ by a linear change of variable in the integral. 
Thus 
\[ \int_2^x|E(t,Z)|\, dt \ll \frac{x^2\log^2(xZ)}{Z} + \frac{x}{\log x} \ll \frac{x}{\log x}, \]
and
\[ \int_2^x|E(t,Z)|^2\, dt \ll \frac{x^3\log^4(xZ)}{Z^2} + \frac{x}{\log x} \ll \frac{x}{\log x}. \]
We conclude from \eqref{Estar} that
\[ E^* \ll \delta x^2 + \frac{x}{\log x} \ll \delta x^2, \]
since by \eqref{Zdelta}, $x \ll \delta x^2$. Substituting these estimates into \eqref{calJmain} proves \eqref{Jzeroformula}. 

\end{proof}

\begin{proof}[Proof of \Cref{lem3}] This proof comes from \cite[pp. 126--127]{GV1996}.
We take $1\le h \le x/4$. On letting $t=y+u$ with $0\le u\le h$, we have
\[ \begin{split} J :&=\int_{x/2}^x (f(t+h)-f(t))^2\, dt 
= \int_{x/2-u}^{x-u} (f(y+u+h)-f(y+u))^2\, dy \\ &
= \frac1h\int_0^h \left(\int_{x/2-u}^{x-u} (f(y+u+h)-f(y+u))^2\, dy \right) du \\&
\le \frac1h\int_0^h \left(\int_{x/2-u}^{x-u} \left( |f(y+u+h)-f(y)| +|f(y+u)-f(y)| \right)^2\, dy \right) du \\& 
\le \frac2h\int_0^h \left(\int_{x/2-u}^{x-u} (f(y+u+h)-f(y))^2 +(f(y+u)-f(y))^2 \, dy \right) du . \end{split} \]
Since the integration range of the inner integral always lies in the interval $[x/4,x]$, we have
\[ J \le \frac2h\int_0^{2h} \left(\int_{x/4}^{x} (f(y+u)-f(y))^2 \, dy \right) du = \frac2h \iint\limits_{\mathcal{R}} (f(y+u)-f(y))^2 dA,\]
where $\mathcal{R}$ is the region defined by $x/4\le y \le x$ and $0\le u\le 2h$. Making the change of variable $u=\delta y$ then $\mathcal{R}$ is defined by $x/4\le y \le x$ and $0\le \delta \le 2h/y$, and changing the order of integration gives
\[ \begin{split} J &\le \frac2h \int_{x/4}^{x} \left(\int_0^{2h/y} (f(y+\delta y)-f(y))^2 y \, d\delta \right) \, dy \\ &
\le \frac{2x}h \int_{x/4}^{x} \left(\int_0^{8h/x} (f(y+\delta y)-f(y))^2 \, d\delta \right) \, dy
\end{split} \] 
Inverting the order of integration again, we conclude
\[ J \le \frac{2x}{h} \int_0^{8h/x}\left(\int_0^x(f(y+\delta y)- f(y) )^2 \, dy\right) \, d\delta . \]
\end{proof}

\section{Proof of Theorem \ref{thm6}, Corollary \ref{cor_BHMS}, and Theorem \ref{thm7}.}

\begin{proof}[Proof of Theorem \ref{thm6}] 
From \cite[Theorem 30]{Ingham1932}, it is well-known that 
\[ \psi(x) - x \ll x^{\Theta}\log^2\!x , \]
but it seems less well-known that in 1965 Grosswald \cite{Grosswald} refined this result by proving that, for $1/2<\Theta <1$,
\begin{equation} \label{Grosswald} \psi(x) - x \ll x^{\Theta}, \end{equation}
from which we immediately obtain 
\begin{equation} \label{H-Theta} H(x) := \int_0^x ( \psi(t) - t)^2\, dt \ll x^{2\Theta+1 }. \end{equation}
From \cite[Lemma 8]{Bhowmik-H-M-SGoldbach2019}
we have from the case $q=1$ that, for $x\ge 2$ and $1\le h\le x$,
\begin{equation}\label{BHMS} \int_x^{2x} ( \psi(t+h) -\psi(t) - h)^2\, dt \ll h x^{2\Theta}\log^4\!x.\end{equation}
We first need to prove that the same bound holds for $J(x,h)$. We have
\[ \begin{split} J(x,h) &= \int_0^{h} ( \psi(t+h) -\psi(t) - h)^2\, dt + \int_h^{x} ( \psi(t+h) -\psi(t) - h)^2\, dt \\ &
:= J_1(h) +J_2(x,h). \end{split}\]
For $J_1(h)$ we use \eqref{H-Theta} to see that, for $1\le h\le x$,
\[ J_1(h) \ll \int_0^h (\psi(t+h) - (t+h))^2 \, dt + \int_0^h (\psi(t) -t)^2\, dt \ll H(2h) \ll h^{2\Theta +1} \ll h x^{2\Theta}. \]
For $J_2(x,h)$ we apply \eqref{BHMS} and find for any interval $(x/2^{k+1}, x/2^{k}]$ contained in $[h/2,x]$
\[ \int_{x/2^{k+1}}^{x/2^{k}}( \psi(t+h) -\psi(t) - h)^2\, dt \ll \frac{h x^{2\Theta} \log^4\!x}{2^{2k\Theta}},\]
and summing over $k\ge 0$ to cover the interval $[h,x]$ we obtain
$ J_2(x,h) \ll h x^{2\Theta}\log^4\!x.$ Combining these estimates we conclude, as desired,
\begin{equation}\label{J-Theta} J(x,h) \ll h x^{2\Theta}\log^4\!x , \end{equation}
and using \eqref{H-Theta} and \eqref{J-Theta} in Theorem \ref{thm2} gives $\mathcal{E}(N) \ll N^{2\Theta}\log^5\!N$.
\end{proof} 

\begin{proof}[Proof of Corollary \ref{cor_BHMS}] Using \eqref{specialcase} of Theorem \ref{thm1}, we see, since the sum over zeros is absolutely convergent, that the contribution from zeros with $\beta <1/2$ is $o(x^{3/2})$. 
\end{proof}

\begin{proof}[Proof of Theorem \ref{thm7}] We have if $n$ is odd that
\[\psi_2(n) = \sum_{m+m'=n}\Lambda(m)\Lambda(m') = 2\log 2\sum_{\substack{n=2^j +m \\ j\ge 1}}\Lambda(m) = 2\log 2 \sum_{j\leq \log_2\!n}\Lambda(n-2^j)\]
and therefore
\[\sum_{\substack{n\le N \\ n \ {\rm odd}}} \psi_2(n) = 2\log 2 \sum_{j\le \log_2\!N}\sum_{\substack{2^j< n\le N \\ n \ {\rm odd}}}\Lambda(n-2^j).\]
We may drop the condition that $n$ is odd in the last sum with an error of $O(\log^2\!N)$ since the only non-zero terms in the sum when $n$ is even are when $n$ is a power of 2. Hence
\[\sum_{\substack{n\le N \\ n \ {\rm odd}}} \psi_2(n) = 2\log 2 \sum_{j\le \log_2\!N}\left(\psi(N-2^j) +O(\log^2\!N)\right).\]
Applying the prime number theorem with a modest error term we have
\[\sum_{\substack{n\le N \\ n \ {\rm odd}}} \psi_2(n) = 2\log 2 \sum_{j\le \log_2\!N}\left((N-2^j) +O(N/\log N)\right) = 2N\log N +O(N).\]
\end{proof}


\medskip


\begin{thebibliography}{ABCD199}

\bibitem[BS10]{BhowPuchta2010} G. Bhowmik, J. -C. Schlage-Puchta, {\it Mean representation number of integers as the sum of primes}, Nagoya Math. J., {\bf 200} (2010), 27--33.

\bibitem[BR18]{BhowRuzsa2018} G. Bhowmik, I. Z. Ruzsa, {\it Average Goldbach and the quasi-Riemann hypothesis}, Anal. Math. {\bf 44} (2018), no. 1, 51--56. 

\bibitem[BHMS19]{Bhowmik-H-M-SGoldbach2019} Gautami Bhowmik, Karin Halupczok, Kohji Matsumoto, Yuta Suzuki, {\it Goldbach representations in arithmetic progressions and zeros of Dirichlet L-functions}, Mathematika {\bf 65} (2019), no. 1, 57--97.
 
\bibitem[BH20]{BhowHalup2020}Gautami Bhowmik and Karin Halupczok, {\it Asymptotics of Goldbach Representations}, Proceedings of Various Aspects of Multiple Zeta Functions — in honor of Professor Kohji Matsumoto’s 60th birthday, Advanced Studies in Pure Mathematics {\bf 84} (2020), 1--21.

\bibitem[BKP19]{BKP} J\"org Br\"udern, Jerzy Kaczorowski, Alberto Perelli, {\it Explicit formulae for averages of Goldbach representations}, Trans. Amer. Math. Soc. {\bf 372} (2019), no. 10, 6981--6999
 
\bibitem[Cra21]{Cramer21} H. Cram\'er, {\it Some theorems concerning prime numbers}, Arxiv f{\"o}r Mat. Astr. Fys. {\bf 15} (1921), no. 5, 33pp.


\bibitem[EM07]{EgamiMatsumoto} S. Egami AND K. Matsumoto, {\it Convolutions of the von Mangoldt function and related Dirichlet series}, Proceedings of the 4th China-Japan Seminar held at Shandong, 1--23, S. Kanemitsu and J. -Y. Liu eds., Ser. Number Theory Appl., 2, World Sci. Publ., Hackensack, NJ, 2007.



\bibitem[Fuj91a]{Fujii1} A. Fujii, {\it An additive problem of prime numbers}, Acta Arith. {\bf 58} (1991), 173--179.

\bibitem[Fuj91b]{Fujii2} A. Fujii, {\it An additive problem of prime numbers. II}, Proc. Japan Acad. ser. A Math. Sci. {\bf 67} (1991), 248--252.

\bibitem[Fuj91c]{Fujii3} A. Fujii, {\it An additive problem of prime numbers. III}, Proc. Japan Acad. ser. A Math. Sci. {\bf 67} (1991), 278--283.

\bibitem[Gol05]{Gold2005} D. A. Goldston, {\it Notes on pair correlation of zeros and prime numbers}, Recent Perspectives in Random Matrix Theory and Number Theory, LMS Lecture Note Series {\bf 322}, Edited by F. Mezzadri and N. C. Snaith, Cambridge Unversity Press, 2005, 79--110.

\bibitem[GM87]{GoldMont} D. A. Goldston and H. L. Montgomery, {\it Pair correlation of zeros and primes in short intervals}, Analytic Number Theory and Diophantine Problems (A. C. Adolphson and et al., eds.), Proc. of a Conference at Oklahoma State University (1984), Birkhauser
Verlag, 1987, 183--203.

\bibitem[GS22]{GS22} D. A. Goldston and Ade Irma Suriajaya, {\it The Prime Number Theorem and Pair Correlation of Zeros of the Riemann Zeta-Function}, Res. Number Theory {\bf 8} (2022), article number: 71, doi: 10.1007/s40993-022-00371-4.

\bibitem[GV96]{GV1996} D. A. Goldston and R. C. Vaughan, {\it On the Montgomery-Hooley asymptotic formula}, Sieve Methods, Exponential Sums, and their Application in Number Theory, Greaves, Harman,
Huxley Eds., Cambridge University Press, (1996), 117--142.

\bibitem[GY17]{Gold-Yang} D. A. Goldston and L. Yang, {\sl The Average Number of Goldbach Representations}, in: Prime numbers and representation theory, Edited by Ye Tian \& Yangbo Ye, Science Press, Beijing, 2017, 1--12.

\bibitem[Gra07]{Gran1} A. Granville, {\it Refinements of Goldbach's conjecture, and the generalized Riemann hypothesis}, Funct. Approx. Comment. Math. {\bf 37} (2007), 159--173.

\bibitem[Gra08]{Gran2} A. Granville, {\it Corrigendum to \lq\lq Refinements of Goldbach's conjecture, and the generalized Riemann hypothesis\rq\rq}, Funct. Approx. Comment. Math. {\bf 38} (2008), 235--237.

\bibitem[Gro65]{Grosswald} {\'E}mile
Grosswald, {\it Sur l'ordre de grandeur des diff{\'e}rences $\psi(x)-x$ et $\pi(x)-\text{li}\, x$}, 
C. R. Acad. Sci. Paris {\bf 260} (1965), 3813–3816. 



\bibitem[Hea82]{Heath-Brown} D. R. Heath-Brown, {\it Gaps between primes, and the pair correlation of zeros of the zeta-function}, Acta Arith. {\bf 41} (1982), 85--99.

\bibitem[Ing32]{Ingham1932} A. E. Ingham, {\it The Distribution of Prime Numbers}, Cambridge Mathematical Library, Cambridge University Press, Cambridge, 1990. Reprint of the 1932 original; With a foreword by R. C. Vaughan.

\bibitem[vonKoc01]{vonKoch1901} H. von Koch, {\it Sur la distribution des nombres premiers}, Acta Math. {\bf 24} (1901), 159--182.

\bibitem[Kou19]{Kouk2019} Dimitris Koukoulopoulos, {\it The distribution of prime numbers}, Graduate Studies in Mathematics {\bf 203}, American Mathematical Society, Providence, RI, 2019.


\bibitem[LPZ12]{LPZ12} Alessandro Languasco, Alberto Perelli, Alessandro Zaccagnini, {\it Explicit relations between pair correlation of zeros and primes in short intervals}, J. Math. Anal. Appl. {\bf 394} (2012), no. 2, 761--771.

\bibitem[LPZ16]{LPZ16} Alessandro Languasco, Alberto Perelli, Alessandro Zaccagnini, {\it An extension of the pair-correlation conjecture and applications}, Math. Res. Lett. {\bf 23} (2016), no. 1, 201--220.

\bibitem[LPZ17]{LPZ17} A. Languasco, A. Perelli, A. Zaccagnini, {\it An extended pair-correlation conjecture and primes in short intervals}, Trans. Amer. Math. Soc. {\bf 369} (2017), no. 6, 4235--4250.

\bibitem[LZ12]{Lang-Zac1} A. Languasco and A. Zaccagnini, {\it The number of Goldbach representations of an integer}, Proc. Amer. Math. Soc. {\bf 140} (2012), 795--804.

\bibitem[LZ15]{Lang-Zac2015} A. Languasco and A. Zaccagnini, {\it A Ces\`aro average of Goldbach numbers}, Forum
Math. {\bf 27} (2015), no. 4, 1945–-1960,

\bibitem[Lan12]{Landau1912} E. Landau, {\it \"Uber die Nullstellen der Zetafunktion}, Math. Ann. {\bf 71} (1912), no. 4, 548--564.

\bibitem[Mon71]{Montgomery71} Hugh L. Montgomery, {\it Topics in Multiplicative Number Theory}, Lecture Notes in Mathematics, Vol. 227, Springer-Verlag, Berlin-New York, 1971.

\bibitem[Mon72]{Montgomery72} H.L. Montgomery, The Pair Correlation of Zeros of the Zeta Function, in: Analytic Number Theory, Proc. Sympos. Pure Math., Vol. XXIV, St. Louis Univ., St. Louis, Mo., 1972, 181--193.

\bibitem[MV73]{MontgomeryVaughan1973}H. L. Montgomery and R. C. Vaughan, {\it Error terms in additive prime number theory}, Quart. J. Math. Oxford (2), {\bf 24} (1973), 207--216.


\bibitem[MV07]{MontgomeryVaughan2007} H. L. Montgomery and R. C. Vaughan, {\it Multiplicative Number Theory}, Cambridge Studies in Advanced Mathematics {\bf 97}, Cambridge University Press, Cambridge, 2007.

\bibitem[Sel43]{Selberg} A. Selberg, {\it On the normal density of primes in small intervals, and the difference between consecutive primes}, Arch. Math. Naturvid. {\bf 47} (1943), 87--105.

\bibitem[SV77]{SaffariVaughan} B. Saffari and R. C. Vaughan, {\it On the fractional parts of $x/n$ and related sequences II}, Ann. Inst. Fourier (Grenoble) {\bf 27} (1977), no. 2, 1--30.

\bibitem[PT21]{PlattTrudgian} D. J. Platt and T. S. Trudgian, {\it The error term in the prime number theorem}, Math. Comp. {\bf 90} (2021), 871-881.

\bibitem[Tit86]{Titchmarsh} E. C. Titchmarsh, {\sl The Theory of the Riemann Zeta-Function}, 
2nd ed., revised by D. R. Heath-Brown, Clarendon (Oxford), 1986.


\end{thebibliography}
\end{document}